\documentclass{amsart}
\usepackage{amsmath, amsthm, amssymb, mathtools, mathrsfs, stmaryrd}
\DeclareSymbolFont{stmry}{U}{stmry}{m}{n}
\SetSymbolFont{stmry}{bold}{U}{stmry}{m}{n}
\usepackage{ascmac}
\usepackage{enumitem}
\usepackage{comment}
\usepackage{bm}
\allowdisplaybreaks

\usepackage{graphicx}
\usepackage[top=25mm, bottom=25mm, left=24mm, right=24mm]{geometry}
\usepackage{hyperref}
\usepackage[abbrev]{amsrefs}
\usepackage{xcolor}
\usepackage[capitalize,nameinlink,noabbrev,nosort]{cleveref}

\makeatletter
\@namedef{subjclassname@2020}{\textup{2020} Mathematics Subject Classification}
\makeatother

\newtheorem{theoremcounter}{Theorem Counter}[section]

\theoremstyle{remark}
\newtheorem{remark}[theoremcounter]{Remark}

\theoremstyle{definition}
\newtheorem{definition}[theoremcounter]{Definition}

\theoremstyle{plain}
\newtheorem{lemma}[theoremcounter]{Lemma}
\newtheorem{proposition}[theoremcounter]{Proposition}

\newtheorem{theorem}[theoremcounter]{Theorem}

\numberwithin{equation}{section}

\newcommand{\Z}{\mathbb{Z}}

\newcommand{\C}{\mathbb{C}}

\newcommand{\bbH}{\mathbb{H}}
\newcommand{\calU}{\mathcal{U}}
\newcommand{\frakG}{\mathfrak{G}}

\DeclareMathOperator{\ImNew}{Im}
\renewcommand{\Im}{\ImNew}
\DeclareMathOperator{\ReNew}{Re}
\renewcommand{\Re}{\ReNew}

\DeclareMathOperator{\SL}{SL}

\newcommand{\st}[2]{{\genfrac{[}{]}{0pt}{}{#1}{#2}}}

\newcommand{\ol}[1]{\overline{#1}}
\newcommand{\abs}[1]{\lvert#1\rvert}
\newcommand{\modN}{\ (\mathrm{mod}\, N)}

\begin{document}
  \title[A Generalization of {M}ac{M}ahon Series via Cyclotomic Polynomials]{A Generalization of {M}ac{M}ahon Series via Cyclotomic Polynomials}
  \author[Riku Shintani]{Riku Shintani}
  \address{Graduate School of Science and Engineering, Ehime University, 2-5, Bunkyo-cho, Matsuyama, Ehime, 790-8577, Japan}
  \email{m807024z@mails.cc.ehime-u.ac.jp}



  \begin{abstract}
    About a century ago, P. A. MacMahon introduced a class of $q$-series, which are nowadays referred to as MacMahon series. More recently, in 2013, G. E. Andrews and S. C. F. Rose revealed the quasimodular property of these series. In this paper, we introduce a generalization of MacMahon series. Specifically, for any positive integers $t, k, N$ and a polynomial $Q(x)$,  we define the series $\mathcal{U}_{t, k; N}(Q; q)$ and $\mathcal{U}_{t, k; N}^{\star}(Q; q)$ using the $N$-th cyclotomic polynomial. To investigate these series, we apply a decomposition formula involving the Eulerian polynomials and express the $N$-th roots of unity in terms of Gauss sums. By combining these results to derive explicit representations, we prove that our series arise as quasimodular forms of higher weight and higher level. Furthermore, we show that they can be expressed as isobaric polynomials. In particular, we show that the one-parameter generalization introduced by C. Nazaroglu, B. V. Pandey, and A. Singh arises as a special case of our theory.
  \end{abstract}
  \maketitle

  \section{Introduction}\label{Intro}
  In \cite{MacMahonDiv}, MacMahon introduced a class of $q$-series, for $t \in \Z_{\geq 1}$, defined by 
  \[
  \calU_{t}(q) \coloneq \sum_{1 \leq n_1 < n_2 < \cdots < n_t}
  \prod_{j = 1}^{t}\frac{q^{n_j}}{(1-q^{n_j})^2}.
  \]
  These sums have links to various fields, such as number theory and combinatorics (see e.g. \cites{Yuan2016, BachmannBialgeb}). In particular, they are known to deep connections to the theory of modular forms and quasimodular forms, since they can be expressed as a linear combination of quasimodular forms of weights up to $2t$ on $\SL_2(\Z)$ as shown by Andrews and Rose \cites{AndRose, Rose}.

  Various analogues of $\calU_t(q)$ have recently researched, especially with regard to modular properties (see e.g., \cite{allied_q-series}). Our investigation is motivated by a related class of series, introduced for $t, k, r \in \Z_{\geq 1}$ by Nazaroglu, Pandey, and Singh \cite{nazaroglu2025}
  \[
  \calU_{t, k, r}(a; q) \coloneq \sum_{1 \leq n_1 < n_2 < \cdots < n_t}
  \prod_{j = 1}^{t}\frac{q^{r n_j}}{(1 + a q^{n_j} + q^{2 n_j})^k}.
  \]
  In {\cite{nazaroglu2025}*{Theorem 1.1}}, they showed that for $a \in \{0, \pm 1, 2\}$ and any $t, k \in \Z_{\geq 1}$, each $q$-series $\calU_{t, k, k}(a; q)$ is a mixed-weight quasimodular form on suitable congruence subgroups.

  As a generalization, for $t, k, N \in \Z_{\geq 1}$ and some polynomial $Q(x)$, we define
  \begin{equation}\label{CMS}
    \calU_{t, k; N}(Q; q) \coloneq \sum_{1 \leq n_1 < n_2 < \cdots < n_t}\prod_{j = 1}^{t}\frac{Q(q^{n_j})}{\Phi_N(q^{n_j})^k},
  \end{equation}
  where $\Phi_N(x)$ denotes the $N$-th cyclotomic polynomial. We also define a variant of \eqref{CMS}
  \[
  \calU_{t, k; N}^{\star}(Q; q) \coloneq \sum_{1 \leq n_1 \leq n_2 \leq \cdots \leq n_t}\prod_{j = 1}^{t}\frac{Q(q^{n_j})}{\Phi_N(q^{n_j})^k}.
  \]
  In particular, we define 
  \[
  \calU_{k; N}(Q; q) \coloneq \calU_{1, k; N}(Q; q) = \calU_{1, k; N}^{\star}(Q; q) = \sum_{n \geq 1}\frac{Q(q^n)}{\Phi_N(q^n)^k}.
  \]

  We note that $\calU_{t, k; N}(Q; q)$ recovers $\calU_{t, k, r}(a; q)$ for suitable choices of parameters. For instance, setting $Q(x) = x^r$, $\calU_{t, k, r}(1; q) = \calU_{t, k; 3}(Q; q),\ \calU_{t, k, r}(0; q) = \calU_{t, k; 4}(Q; q)$, and $\calU_{t, k, r}(-1; q) = \calU_{t, k; 6}(Q; q)$.

  As our main theorem, we have the following.

  \begin{theorem}\label{main_thm}
    Let $t, k, N \in \Z_{\geq 1}$ and $Q(x)$ be a polynomial of degree less than $\phi(N)k$ satisfying $Q(0) = 0$ and 
    \[
    Q(x) = 
      \begin{cases}
        (-x)^k Q(1/x) & \text{if } N = 1, \\
        x^{\phi(N)k} Q(1/x) & \text{if } N \geq 2,
      \end{cases}
    \]
    where $\phi(N)$ denotes Euler's totient function. Then, we have the following:
    \begin{enumerate}[label = (\arabic*)]
      \item\label{quasimodularity_main} $\calU_{k; N}(Q; q)$ is a mixed-weight quasimodular form of weight at most $k$ and level at most $N$.
      \item\label{decomposition_main} $\calU_{t, k; N}(Q; q)$ and $\calU_{t, k; N}^{\star}(Q; q)$ are isobaric polynomial of weight $t$ in the variables $\calU_{sk; N}(Q^s; q)$ for $1 \leq s \leq t$. That is, each term is a monomial of the form 
      \[
      \calU_{k; N}(Q; q)^{s_1} \calU_{2k; N}(Q^2; q)^{s_2} \cdots \calU_{tk; N}(Q^t; q)^{s_t}
      \]
      satisfying $s_1 + 2s_2 + \cdots + ts_t = t$. In other words, both of these are mixed-weight quasimodular form of weight at most $tk$ and level at most $N$.
    \end{enumerate}
  \end{theorem}

  Throughout this paper, we denote the complex upper half-plane by $\bbH$ i.e., $\bbH \coloneq \{z \in \C : \Im(z) > 0\}$, and we set $q \coloneq e^{2\pi i \tau}$ for $\tau \in \bbH$.

  The organization of this paper is  as follows. In Section \ref{Prelim}, we introduce preliminary notions and related results on Eisenstein series, Eulerian polynomials, and Gauss sums. In Section \ref{Prf_of_main1}, we apply these results to prove Theorem \ref{main_thm} \ref{quasimodularity_main}. In Section \ref{Prf_of_main2}, we show Theorem \ref{main_thm} \ref{decomposition_main} using ideas from combinatorics. In Appendix \ref{appedix}, we give explicit formulas for the coefficients appearing in Section \ref{Prf_of_main1}. Using these formulas, we recover the results in {\cite{nazaroglu2025}*{Example 1.2}}.
  
  \section{Preliminaries}\label{Prelim}
  \subsection{Eisenstein Series and Quasimodular Forms}
  \begin{definition}
    Let $\chi$ be a Dirichlet character modulo $N \in \Z_{\geq 1}$ and $k \in \Z_{\geq 1}$ such that $\chi(-1) = (-1)^k$.
    Then for $\tau \in \bbH$ and $s \in \C$ with $k + 2 \Re(s) > 2$, we define the \emph{non-holomorphic Eisenstein series}
    \[
    \frakG_{k}(\chi; \tau; s) \coloneq \frac{1}{2}\sum_{\substack{(m, n) \in \Z^2 \setminus \{(0, 0)\} \\ m \equiv 0 \modN}}\frac{\ol{\chi(n)}}{(m\tau + n)^k}\frac{\Im(\tau)^s}{\abs{m\tau + n}^{2s}}.
    \]
  \end{definition}
  
  The non-holomorphic Eisenstein series $\frakG_k(\chi; \tau; s)$ has a meromorphic continuation to the whole complex plane. If $\chi$ is a non-trivial character, it is an entire function. However, for the trivial character $\chi = \mathbf{1}$, it has a simple pole at $s = 1 - k/2$. This pole does not lie at $s = 0$ for $k \geq 4$, and for $k = 2$ we use a procedure known as Hecke's trick (see e.g. {\cite{123Modular}*{Proposition 6}}) to define the value at $s = 0$. We define the holomorphic Eisenstein series as the specialization of $\frakG_k(\chi; \tau; s)$ at $s = 0$, normalized by a constant factor so that its Fourier expansion is given by the following lemma. We denote this function by $G_k(\chi; \tau)$ if $\chi \neq \mathbf{1}$ and simply by $G_k(\tau)$ if $\chi = \mathbf{1}$.
  
  \begin{lemma}{\cite{nazaroglu2025}*{Lemma 2.1}}\label{Eisen_explict}
    The Fourier expansions are given as follows. 
    \begin{enumerate}[label = (\arabic*)]
      \item Let $N \in \Z_{\geq 3}$, $k \in \Z_{\geq 1}$, and $\chi$ be a primitive Dirichlet character modulo $N$ satisfying $\chi(-1) = (-1)^k$. Then, we have
      \[
        G_k(\chi; \tau) = -\frac{B_{k, \chi}}{2k} + F_k(\chi; \tau),
      \]
      where $B_{k, \chi}$ denotes the $k$-th generalized Bernoulli number defined by 
      \[
        \sum_{a = 1}^{N}\frac{\chi(a)te^{at}}{e^{Nt} - 1} = \sum_{k \geq 0}B_{k, \chi}\frac{t^k}{k!},
      \]
      and 
      \[
        F_k(\chi; \tau) \coloneq \sum_{m, n \geq 1}\chi(m)m^{k-1}q^{mn}.
      \]
      \item Let $\chi = \mathbf{1}$ be the trivial character modulo $1$. For an even $k \geq 2$, we have
      \[
        G_k(\tau) = -\frac{B_k}{2k} + F_k(\tau),
      \]
      where $B_k \coloneq B_{k, \mathbf{1}}$ denotes the $k$-th Bernoulli number and 
      \[
        F_k(\tau) \coloneq \sum_{m, n \geq 1}m^{k-1}q^{mn}.
      \]
    \end{enumerate}
  \end{lemma}

  For even $k \geq 4$, we know that $G_{k}(\tau) \in M_{k}(\SL_{2}(\Z))$, where $M_{k}(\SL_{2}(\Z))$ is the space of holomorphic modular forms of wight $k$ on $\SL_{2}(\Z)$. In contrast, $G_{2}(\tau)$ is a quasimodular form on $\SL_2(\Z)$. More specifically, a holomorphic function $f: \bbH \to \C$ is called a quasimodular form of weight $k$ and character $\chi$ on $\Gamma_0(N)$ if it satisfies the following properties:
  \begin{itemize}
    \item There exists $s \in \Z_{\geq 0}$ and holomorphic functions $f_i: \bbH \to \C\ (0 \leq i \leq s)$ such that for any $\begin{bsmallmatrix} a & b \\ c & d \end{bsmallmatrix} \in \Gamma_0(N)$, we have
    \[
    f\left(\frac{a\tau + b}{c\tau + d}\right) = \chi(d)(c\tau + d)^k\sum_{i = 0}^{s}f_i(\tau)\left(\frac{c}{c\tau + d}\right)^i.
    \]
    \item $f$ is holomorphic at all cusps of $\Gamma_0(N)$.
  \end{itemize}
  We denote the space of such functions by $\widetilde{M}_k(\Gamma_0(N), \chi)$.

  \begin{lemma}{\cite{WebMod}*{Proposition 2.22}}\label{V_op}
    Let $f(\tau) \in M_k(\Gamma_0(N), \chi)$. Then, for any $K \in \Z_{\geq 1}$, we have $f(K\tau) \in M_k(\Gamma_0(KN), \chi)$. This property also holds for $f(\tau) \in \widetilde{M}_{k}(\Gamma_0(N), \chi)$.
  \end{lemma}

  \subsection{Eulerian Polynomials}

  \begin{definition}
    For $k \in \Z_{\geq 0}$, $k$-th \emph{Eulerian polynomial} $P_{k}(x)$ is defined by 
    \begin{equation}\label{Eulerian}
      \frac{P_{k}(x)}{(1 - x)^{k+1}} \coloneq \sum_{n \geq 1}n^k x^{n-1}
    \end{equation}
  \end{definition}

  The first few are given by
  \[
  P_{0}(x) = 1,\ P_{1}(x) = 1,\ P_{2}(x) = 1 + x,\ P_{3}(x) = 1 + 4x + x^2.
  \]
  For $k \in \Z_{\geq 1}$, the polynomial $P_{k}(x)$ has degree $k - 1$ and is reciprocal, that is, it satisfies the relation
  \begin{equation}\label{reciprocity}
    P_{k}(x) = x^{k-1}P_{k}(1/x).
  \end{equation}
  The following lemma provides a decomposition involving the Eulerian polynomials.

  \begin{lemma}\label{PFD_lemma}
    For $r \in \Z_{\geq 1}$, we have 
    \begin{equation}\label{pfd}
    \frac{1}{(1 - x)^{r}} = \frac{1}{(r-1)!}\sum_{\ell = 1}^{r}\st{r-1}{\ell - 1}\frac{P_{\ell - 1}(x)}{(1 - x)^{\ell}},
    \end{equation}
    where $\st{n}{k}$ denotes the unsigned Stirling number of the first kind, defined using the rising factorial $(x)_n \coloneq x(x+1)\cdots(x+n-1)$ by
    \begin{equation}\label{def_of_st}
      (x)_n = \sum_{k = 0}^{n}\st{n}{k}x^k.
    \end{equation}
  \end{lemma}
  \begin{proof}
    From \eqref{Eulerian}, the right-hand side of \eqref{pfd} can be rewritten as 
    \begin{equation}\label{rewrite_RHS_of_PFD}
      \frac{1}{(r-1)!}\sum_{\ell = 1}^{r}\st{r-1}{\ell-1}\left(\sum_{m \geq 1}m^{\ell-1}x^{m-1}\right) = \frac{1}{(r-1)!}\sum_{m \geq 0}\left(\sum_{\ell = 1}^{r}\st{r-1}{\ell-1}(m+1)^{\ell-1}\right)x^m.
    \end{equation}
    Here, using \eqref{def_of_st}, we see that
    \begin{align*}
      \sum_{\ell = 1}^{r}\st{r-1}{\ell-1}(m+1)^{\ell-1} 
      &= \sum_{\ell = 0}^{r-1}\st{r-1}{\ell}(m+1)^{\ell} \\
      &= (m+1)_{r-1} \\
      &= \frac{(m+r-1)!}{m!}.
    \end{align*}
    Substituting this, the right-hand side of \eqref{rewrite_RHS_of_PFD} equals
    \[
      \frac{1}{(r-1)!}\sum_{m \geq 0}\frac{(m+r-1)!}{m!}x^m
      = \sum_{m \geq 0}\binom{r+m-1}{m}x^m
      = \frac{1}{(1 - x)^r}.
      \qedhere
    \]
  \end{proof}

  \subsection{Gauss Sums}

  \begin{definition}\label{Gauss_sum}
    For a Dirichlet character modulo $N \in \Z_{\geq 1}$, we define the \emph{Gauss sum} associated with $\chi$ by
    \[
    G(\chi) \coloneq \sum_{a = 1}^{N}\chi(a)\zeta_{N}^a,
    \]
    where $\zeta_N \coloneq e^{2\pi i/N}$.
  \end{definition}
  To prove Theorem \ref{main_thm} \ref{quasimodularity_main}, we use the following expression for powers of roots of unity.
  
  \begin{lemma}\label{power_roots}
    Let $N, m \in \Z_{\geq 1}$ and $g \coloneq (N, m)$. Then, we have
    \begin{equation}\label{power_of_roots_of_unity}
      \zeta_{N}^m = \frac{1}{\phi(N/g)}\sum_{\chi\mathrm{: mod \, } N/g}G(\chi)\ol{\chi(m/g)},
    \end{equation}
    where the sum runs over all Dirichlet characters $\chi$ modulo $N/g$.
  \end{lemma}
  \begin{proof}
    Let $N' = N/g$ and $m' = m/g$. Then $(N', m') = 1$ and $\zeta_N^m = \zeta_{N'}^{m'}$. Using these notations and the definition of the Gauss sum,  the right-hand side of \eqref{power_of_roots_of_unity}  can be rewritten as
    \begin{align*}
      \frac{1}{\phi(N')}\sum_{\chi\mathrm{: mod \, }N'}G(\chi)\ol{\chi(m')}
      &= \frac{1}{\phi(N')}\sum_{\chi\mathrm{: mod \, }N'}\left(\sum_{a = 1}^{N'}\chi(a)\zeta_{N'}^{a}\right) \ol{\chi(m')} \\
      &= \frac{1}{\phi(N')}\sum_{a = 1}^{N'}\zeta_{N'}^{a}\sum_{\chi\mathrm{: mod \, }N'}\chi(a)\ol{\chi(m')} \\
      &= \frac{1}{\phi(N')}\sum_{\substack{a = 1 \\ (a, N') = 1}}^{N'}\zeta_{N'}^{a}\sum_{\chi\mathrm{: mod \, }N'}\chi(am'^{-1}) \\
      &= \zeta_{N'}^{m'}.
    \end{align*}
    In the last equality, we have used the orthogonality relation for Dirichlet characters
    \[
    \sum_{\chi\mathrm{: mod \, }N} \chi(a) = 
    \begin{cases}
      \phi(N) & \text{if } a \equiv 1 \modN , \\
      0 & \text{if } a \not\equiv 1 \modN.
    \end{cases}
    \qedhere
    \]
  \end{proof}

  \section{Proof of Theorem \ref{main_thm} \ref{quasimodularity_main}}\label{Prf_of_main1}
  In this section, we prove Theorem \ref{main_thm}\ref{quasimodularity_main} by providing an explicit formula for $\calU_{k; N}(Q; q)$.

  \subsection{The Case \texorpdfstring{$N = 1$}{N = 1}}
  
  \begin{theorem}\label{main_N=1}
  Under the notation and assumption of Theorem \ref{main_thm}, we have
  \begin{equation}\label{N=1}
    \calU_{k; 1}(Q; q) = \sum_{\substack{\ell = 1 \\ \ell\mathrm{:even}}}^{k}c_{1; Q}(\ell)F_{\ell}(\tau),
  \end{equation}
  where the coefficients $c_{1; Q}(\ell)$ are defined by 
  \[
  c_{1; Q}(\ell) \coloneq \sum_{r = \ell}^{k}\frac{a_{1; Q}(r)}{(r-1)!}\st{r-1}{\ell-1}
  \]
  with the coefficients $a_{1; Q}(r)$ uniquely determined by the partial fraction decomposition of $\frac{Q(x)/x}{\Phi_1(x)^k}$
  \begin{equation}\label{PFD_N=1}
    \frac{Q(x)/x}{\Phi_1(x)^k} = \sum_{r = 1}^{k}a_{1; Q}(r)\frac{1}{(1-x)^r}.
  \end{equation}
  In particular, $\calU_{k; 1}(Q; q)$ is a linear combination of quasimodular forms of level $1$ and weight at most $k$.
  \end{theorem}
  \begin{proof}
    It is sufficient to prove \eqref{N=1}. We start from the decomposition given in \eqref{PFD_N=1}. After multiplying by $x$ and applying Lemma \ref{PFD_lemma}, we obtain
    \begin{align}\label{before_transform_N=1}
      \frac{Q(x)}{\Phi_1(x)^k} 
      &= \sum_{r = 1}^{k}\frac{a_{1; Q}(r)}{(r-1)!}\sum_{\ell = 1}^{r}\st{r-1}{\ell-1}\frac{x P_{\ell-1}(x)}{(1-x)^{\ell}} \notag \\
      &= \sum_{\ell = 1}^{k}c_{1; Q}(\ell)\frac{x P_{\ell-1}(x)}{(1-x)^{\ell}} \notag \\
      &= \sum_{\ell = 1}^{k}c_{1; Q}(\ell) \sum_{m \geq 1}m^{\ell-1}x^m.
    \end{align}

    We consider the transformation $x \mapsto 1/x$. By our assumption on $Q(x)$ and the $\Phi_1(x)^k = (-x)^k \Phi_1(1/x)^k$, the left-hand side of \eqref{before_transform_N=1} is invariant under this change. This implies
    \begin{align}\label{after_transform_N=1}
      \frac{Q(x)}{\Phi_1(x)^k} 
      &= \sum_{\ell = 1}^{k}(-1)^{\ell}c_{1; Q}(\ell)\frac{x^{\ell-1}P_{\ell-1}(1/x)}{(1-x)^{\ell}} \notag \\
      &= - c_{1; Q}(1) + \sum_{\ell = 1}^{k}(-1)^{\ell}c_{1; Q}(\ell)\frac{x P_{\ell-1}(x)}{(1-x)^{\ell}} \notag \\
      &= - c_{1; Q}(1) + \sum_{\ell = 1}^{k}(-1)^{\ell}c_{1; Q}(\ell)\sum_{m \geq 1}m^{\ell-1}x^m.
    \end{align}
    Here, we have used the following relation derived from \eqref{reciprocity}
    \[
    \frac{x^{\ell-1}P_{\ell-1}(1/x)}{(1-x)^{\ell}} =
    \delta_{1, \ell} + \frac{x P_{\ell - 1}(x)}{(1 - x)^{\ell}} ,
    \]
    where $\delta_{1, \ell}$ denotes the Kronecker delta, which is $1$ if $\ell = 1$ and $0$ otherwise. The decomposition of the left-hand side into such series is unique. Therefore, by comparing \eqref{before_transform_N=1} with \eqref{after_transform_N=1}, we conclude that 
    \[
    c_{1; Q}(\ell) = 0 \quad \text{for all odd $\ell$ with } 1 \leq \ell \leq k.
    \]
    Hence, we have 
    \begin{equation}\label{sort_N=1}
      \frac{Q(x)}{\Phi_1(x)^k} = \sum_{\substack{\ell = 1 \\ \ell\text{:even}}}^{k}c_{1; Q}(\ell)\sum_{m \geq 1}m^{\ell-1}x^{m}.
    \end{equation}
    Setting $x = q^n$ in \eqref{sort_N=1} and summing over all $n \in \Z_{\geq 1}$,
    \begin{align*}
      \calU_{k; 1}(Q; q) 
      &= \sum_{\substack{\ell = 1 \\ \ell\text{:even}}}^{k} c_{1; Q}(\ell) \sum_{m, n \geq 1}m^{\ell-1}q^{mn} \\
      &= \sum_{\substack{\ell = 1 \\ \ell\text{:even}}}^{k}c_{1; Q}(\ell) F_{\ell}(\tau). 
      \qedhere
    \end{align*}
  \end{proof}

  \begin{remark}
    By the definition of $c_{1; Q}(\ell)$, we have $c_{1; Q}(k) = \frac{(-1)^k Q(1)}{(k-1)!}$. Thus, the series \eqref{N=1} is of weight $k$ if and only if $Q(1) \neq 0$ and $k$ is even.
  \end{remark}

  \subsection{The Case \texorpdfstring{$N = 2$}{N = 2}}
  \begin{theorem}\label{main_N=2}
    Under the notation and assumptions of Theorem \ref{main_thm}, we have 
    \begin{equation}\label{N=2}
      \calU_{k; 2}(Q; q) = \sum_{\substack{\ell = 1 \\ \ell\mathrm{:even}}}^{k}
      c_{2; Q}(\ell) (2^{\ell}F_{\ell}(2\tau) - F_{\ell}(\tau)),
    \end{equation}
    where the coefficients $c_{2; Q}(\ell)$ are defined by 
    \[
    c_{2; Q}(\ell) \coloneq \sum_{r = \ell}^{k}\frac{a_{2; Q}(r)}{(r-1)!}\st{r-1}{\ell-1}
    \]
    with the coefficients $a_{2; Q}(r)$ uniquely determined by the partial fraction decomposition of $\frac{Q(x)/x}{\Phi_2(x)^k}$
    \[
    \frac{Q(x)/x}{\Phi_2(x)^k} = \sum_{r = 1}^{k}a_{2; Q}(r)\frac{-1}{(1 + x)^r}.
    \]
    In particular, $\calU_{k; 2}(Q; q)$ is a linear combination of quasimodular forms of level $2$ and weight at most $k$.
  \end{theorem}
  \begin{proof}
    In a similar way to the proof of Theorem \ref{main_N=1}, we have
    \begin{align*}
      \calU_{k; 2}(Q; q) 
      &= \sum_{\substack{\ell = 1 \\ \ell\text{:even}}}^{k} c_{2; Q}(\ell)\sum_{m, n \geq 1}(-1)^m m^{\ell-1}q^{mn} \\
      &= \sum_{\substack{\ell = 1 \\ \ell\text{:even}}}^{k} c_{2; Q}(\ell)\sum_{m, n \geq 1}(2\delta_{2 \mid m} - 1)m^{\ell-1}q^{mn} \\
      &= \sum_{\substack{\ell = 1 \\ \ell\text{:even}}}^{k}c_{2; Q}(\ell)(2^{\ell}F_{\ell}(2\tau) - F_{\ell}(\tau)).
    \end{align*}
    Here, we have used the identity $(-1)^m = 2\delta_{2 \mid m} - 1$ to derive this expression, where $\delta_{2 \mid m}$ is $1$ if $2 \mid m$ and $0$ otherwise. Since $F_{\ell}(2\tau)$ is of level $2$ by Lemma \ref{V_op}, we conclude that $2^{\ell}F_{\ell}(2\tau) - F_{\ell}(\tau)$ is of level $2$.
  \end{proof}
  \begin{remark}
    By the definition of $c_{2; Q}(\ell)$, we have $c_{2; Q}(k) = \frac{Q(-1)}{(k-1)!}$. Thus, the series \eqref{N=2} is of weight $k$ if and only if $Q(-1) \neq 0$ and $k$ is even.
  \end{remark}

  \subsection{The Case \texorpdfstring{$N \geq 3$}{N ≥ 3}}
  \begin{theorem}\label{main_N=N}
    In addition to the notation and assumptions of Theorem \ref{main_thm}, we assume $N \in \Z_{\geq 3}$. Then, we have
    \begin{equation}\label{N=N}
      \calU_{k; N}(Q; q) = \sum_{\substack{j = 1 \\ (j, N) = 1}}^{\lfloor \frac{N-1}{2} \rfloor}
      \sum_{\ell = 1}^{k}c_{N, j; Q}(\ell)
      \sum_{g \mid N}\frac{2g^{\ell-1}}{\phi(N/g)}
      \sum_{\substack{\chi\mathrm{:mod }N/g \\ \chi(-1) = (-1)^{\ell}}}
      G(\chi)\ol{\chi(j)}F_\ell(\ol{\chi}; g\tau),
    \end{equation}
    where and the coefficients $c_{N, j; Q}(\ell)$ are defined by 
    \[
    c_{N, j; Q}(\ell) \coloneq \sum_{r = \ell}^{k}\frac{a_{N, j; Q}(r)}{(r-1)!}\st{r-1}{\ell-1}
    \]
    with the coefficients $a_{N, j; Q}(r)$ uniquely determined by the partial fraction decomposition of $\frac{Q(x)/x}{\Phi_N(x)^k}$
    \[
    \frac{Q(x)/x}{\Phi_N(x)^k} = 
    \sum_{\substack{j = 1 \\ (j, N) = 1}}^{\lfloor \frac{N-1}{2} \rfloor}
    \left(
    \sum_{r = 1}^{k}a_{N, j; Q}(r)\frac{\zeta_N^j}{(1 - \zeta_N^j x)^r} + 
    \sum_{r = 1}^{k}a'_{N, j; Q}(r)\frac{\zeta_N^{-j}}{(1 - \zeta_N^{-j} x)^r}
    \right).
    \]
  \end{theorem}
  \begin{proof}
    In a similar way to the proof of Theorem \ref{main_N=1}, we have 
    {\small \[
    \calU_{k; N}(Q; q) = 
    \sum_{\substack{j = 1 \\ (j, N) = 1}}^{\lfloor \frac{N-1}{2} \rfloor}
    \left(
    \sum_{\substack{\ell = 1 \\ \ell\text{:odd}}}^{k}c_{N, j; Q}(\ell)
    \sum_{m, n \geq 1}(\zeta_N^{jm} - \zeta_{N}^{-jm})m^{\ell-1}q^{mn} + 
    \sum_{\substack{\ell = 1 \\ \ell\text{:even}}}^{k}c_{N, j; Q}(\ell)
    \sum_{m, n \geq 1}(\zeta_N^{jm} + \zeta_{N}^{-jm})m^{\ell-1}q^{mn}
    \right).
    \]}
    Set $N' = N/g$ and $m' = m/g$ with $g \coloneq (N, m)$. Then, from Lemma \ref{power_roots}, we obtain
    \begin{align*}
      \sum_{m, n \geq 1}(\zeta_N^{jm} - \zeta_N^{-jm})m^{\ell-1}q^{mn}
      &= \sum_{g \mid N}\sum_{\substack{m' \geq 1 \\ (m', N') = 1}}\sum_{n \geq 1}
      \left(
      \frac{2}{\phi(N')}\sum_{\substack{\chi\text{:mod }N' \\ \chi\text{:odd}}}
      G(\chi)\ol{\chi(j)}\ol{\chi(m')}
      \right)(gm')^{\ell-1}q^{(gm')n} \\
      &= \sum_{g \mid N}\frac{2g^{\ell-1}}{\phi(N/g)}
      \sum_{\substack{\chi\text{:mod }N/g \\ \chi\text{:odd}}}G(\chi)\ol{\chi(j)}
      \sum_{m', n \geq 1}\ol{\chi(m')}{m'}^{\ell-1}(q^{g})^{m' n} \\
      &= \sum_{g \mid N}\frac{2g^{\ell-1}}{\phi(N/g)}
      \sum_{\substack{\chi\text{:mod }N/g \\ \chi\text{:odd}}}G(\chi)\ol{\chi(j)}
      F_{\ell}(\ol{\chi}; g\tau).
    \end{align*}
    
    By a similar procedure, we also obtain 
    \[
    \sum_{m, n \geq 1}(\zeta_N^{jm} + \zeta_N^{-jm})m^{\ell-1}q^{mn} = 
    \sum_{g \mid N}\frac{2g^{\ell-1}}{\phi(N/g)}
    \sum_{\substack{\chi\text{:mod }N/g \\ \chi\text{:even}}}
    G(\chi)\ol{\chi(j)}F_{\ell}(\ol{\chi};g\tau).
    \]
    Thus, by Lemma \ref{V_op}, we conclude that $F_{\ell}(\chi; g\tau)$ is of level at most $N$.
  \end{proof}

  \begin{remark}
  \leavevmode
  \begin{enumerate}[label = (\arabic*)]
    \item Recall that there exists a primitive Dirichlet character modulo $m$ if and only if $m \not\equiv 2 \pmod{4}$. Thus, the sum \eqref{N=N} always contains a candidate term of level $N$. Indeed, we can choose $g = 1$ if $N \not\equiv 2 \pmod{4}$, and $g = 2$ if $N \equiv 2 \pmod{4}$.
    \item By the definition of $c_{N, j; Q}(\ell)$, we have $c_{N, j; Q}(k) = \frac{1}{(k-1)!}a_{N, j; Q}(k)$. Then, we can evaluate
    \[
    a_{N, j; Q}(k) = \lim_{x \to \zeta_{N}^{-j}}
    \left[
    (1 - \zeta_N^j x)^k \frac{Q(x)}{\Phi_N(x)^k}
    \right] = 
    \left(
    - \frac{1}{N}\prod_{\substack{d \mid N \\ d < N}}\Phi_d(\zeta_N^{-j})
    \right)^k Q(\zeta_N^{-j}).
    \]
    For any $j$ with $(j, N) = 1$ and all $d \mid N$ with $d < N$, we have $\Phi_d(\zeta_N^{-j}) \neq 0$. Thus, the sum \eqref{N=N} is of weight $k$ and level $N$ if and only if $Q(\zeta_N^{-j}) \neq 0$ for some $j$. 
  \end{enumerate}
  \end{remark}
  
  \section{Proof of Theorem \ref{main_thm} \ref{decomposition_main}}
  \label{Prf_of_main2}

  The proof of Theorem \ref{main_thm} \ref{decomposition_main} requires the following lemma.

  \begin{lemma}\label{multisum_decomposition_lemma}
    For $k, N \in \Z_{\geq 1}$ and a polynomial $Q(x)$, we have
    \begin{align}
      \exp\left(\sum_{m \geq 1}\frac{(-1)^{m-1}}{m}\calU_{mk; N}(Q^m; q)X^m \right) &= 1 + \sum_{j \geq 1}\calU_{j, k; N}(Q; q)X^j, \label{non_eq} \\ 
      \exp\left(\sum_{m \geq 1}\frac{1}{m}\calU_{mk; N}(Q^m; q)X^m \right) &= 1 + \sum_{j \geq 1}\calU^{\star}_{j, k; N}(Q; q)X^j. \label{with_eq}
    \end{align}
  \end{lemma}
  \begin{proof}
    We only prove \eqref{non_eq}, since the proof of \eqref{with_eq} is similar. Let $F(X)$ denote the right-hand side of \eqref{non_eq}, which can be rewritten as the following infinite product
    \[
    F(X) = \prod_{n \geq 1}\left(1 + \frac{Q(q^n)}{\Phi_N(q^n)^k}X\right).
    \]
    Taking logarithm of both sides gives 
    \begin{align*}
      \log F(X) 
      &= \sum_{n \geq 1}\log\left(1 + \frac{Q(q^n)}{\Phi_N(q^n)^k}X\right) \\
      &= \sum_{n \geq 1}\left(
      \sum_{m \geq 1}\frac{(-1)^{m-1}}{m}
      \left(
      \frac{Q(q^n)}{\Phi_N(q^n)^k}X
      \right)^m
      \right) \\
      &= \sum_{m \geq 1}\frac{(-1)^{m-1}}{m}\left(
      \sum_{n \geq 1}\frac{Q^m(q^n)}{\Phi_N(q^n)^{mk}}
      \right)X^m \\ 
      &= \sum_{m \geq 1}\frac{(-1)^{m-1}}{m}\calU_{mk; N}(Q^m; q)X^m .
    \end{align*}
    This shows the desired result.
  \end{proof}

  For a formal power series $A(X) = \sum_{m \geq 1}a_m X^m$, we have the exponential formula (see e.g. {\cite{Stanley_EC2}*{Example 5.2.10}})
  \[
  \exp(A(X)) = 1 + \sum_{n \geq 1}\left(\sum_{\lambda \vdash n}\prod_{s = 1}^{n} \frac{a_s^{m_{\lambda, s}}}{m_{\lambda, s}\,!}\right)X^n,
  \]
  where $m_{\lambda, s}$ denotes the multiplicity of $s$ in the partition $\lambda$ of $n$. Thus, setting $a_m = \frac{(-1)^{m-1}}{m}\calU_{mk; N}(Q^m; q)$, we obtain
  \begin{equation}\label{non_eq_decompsition}
    \exp\left(
    \frac{(-1)^{m-1}}{m}\calU_{mk; N}(Q^m; q) X^m
    \right) = 1 + \sum_{n \geq 1}\sum_{\lambda \vdash n}\prod_{s = 1}^{n}
    \frac{1}{m_{\lambda, s}\, !}\left(
    \frac{(-1)^{s-1}\calU_{sk; N}(Q^s; q)}{s}
    \right)^{m_{\lambda, s}} X^n
  \end{equation}
  and setting $a_m = \frac{1}{m}\calU_{mk; N}(Q^m; q)$, we also obtain
  \begin{equation}\label{with_eq_decompsition}
    \exp\left(
    \frac{1}{m}\calU_{mk; N}(Q^m; q) X^m
    \right) = 1 + \sum_{n \geq 1}\sum_{\lambda \vdash n}\prod_{s = 1}^{n}
    \frac{1}{m_{\lambda, s}\, !}\left(
    \frac{\calU_{sk; N}(Q^s; q)}{s}
    \right)^{m_{\lambda, s}} X^n.
  \end{equation}
  Then, by comparing the coefficients the coefficients of $X^t\ (t \geq 1)$ on the right-hand sides of \eqref{non_eq} and \eqref{non_eq_decompsition}, and those of \eqref{with_eq} and \eqref{with_eq_decompsition}, we obtain the desired results
  \begin{align*}
  \calU_{t, k; N}(Q; q) &= \sum_{\lambda \vdash t}\prod_{s = 1}^{t}
    \frac{1}{m_{\lambda, s}\, !}\left(
    \frac{(-1)^{s-1}\calU_{sk; N}(Q^s; q)}{s}
    \right)^{m_{\lambda, s}}, \\ 
  \calU^{\star}_{t, k; N}(Q; q) &= \sum_{\lambda \vdash t}\prod_{s = 1}^{t}
    \frac{1}{m_{\lambda, s}\, !}\left(
    \frac{\calU_{sk; N}(Q^s; q)}{s}
    \right)^{m_{\lambda, s}}.
  \end{align*}

  Since $\sum_{s = 1}^{t}s\, m_{\lambda, s} = t$, it is easy to check that $\calU_{t, k; N}(Q; q)$ and $\calU^{\star}_{t, k; N}(Q; q)$ are isobaric polynomials of weight $t$. The claim about the level follows directly from Theorem \ref{main_thm} \ref{quasimodularity_main}.

  \appendix
  \section{Explicit formulas for \texorpdfstring{$a_{N, j; Q}(r)$}{a\_\{N, j; Q\}(r)} and \texorpdfstring{$c_{N, j; Q}(\ell)$}{c\_\{N, j; Q\}(l)}}\label{appedix}

  We present explicit formulas for the coefficients defined in Section \ref{Prf_of_main1}. Since the other cases can be derived similarly, we restrict our proof to the case $N \geq 3$.

  \begin{proposition}\label{prop:a_1_r}
    Let $k \in \Z_{\geq 1}$ and $Q(x)$ be a polynomial of degree less than $k$ with $Q(0) = 0$. Then, we have 
    \[
      a_{1; Q}(r) = \sum_{m = 0}^{k-r}(-1)^{m + k}\frac{Q^{(m)}(1)}{m!},\ 
      a_{2; Q}(r) = \sum_{m = 0}^{k-r}\frac{Q^{(m)}(-1)}{m!}
    \]
    for all $1 \leq r \leq k$.
  \end{proposition}

  \begin{proposition}
    Under the same assumptions as in Proposition \ref{prop:a_1_r}, we have 
    \[
      c_{1; Q}(\ell) = \frac{1}{(k-1)!}\sum_{r = \ell}^{k}(-1)^{r}Q^{(k-r)}(1)\binom{k-1}{r-1}\st{r}{\ell},\ 
      c_{2; Q}(\ell) = \frac{1}{(k-1)!}\sum_{r = \ell}^{k}Q^{(k-r)}(-1)\binom{k-1}{r-1}\st{r}{\ell}
    \]
    for all $1 \leq \ell \leq k$.
  \end{proposition}

  \begin{proposition}\label{prop:a_N_j_r}
    Let $k \in \Z_{\geq 1}$ and $N \in \Z_{\geq 3}$. Let $Q(x)$ be a polynomial of degree less than $\phi(N)k$ with $Q(0) = 0$. Then, we have
    \begin{equation}\label{a_N_j_Q_r}
      a_{N, j; Q}(r) = \sum_{m = 0}^{k-r}(-\zeta_N^{-j})^m \frac{A_j^{(m)}(\zeta_N^{-j})}{m!}
    \end{equation}
    for all $1 \leq r \leq k$ and $1 \leq j \leq \lfloor (N-1)/2 \rfloor$ such that $(j, N) = 1$. Here, $A_{j}(x)$ is given by 
    \[
      A_j(x) \coloneq (1 - \zeta_N^j x)^k \frac{Q(x)}{\Phi_N(x)^k}.
    \]
  \end{proposition}
  \begin{proof}
    Collecting the terms regular at $x = \zeta_N^{-j}$ and denoting them by $R_j(x)$, we can write
    \[
    \frac{Q(x)}{\Phi_N(x)^k} = \sum_{r = 1}^{k}a_{N, j; Q}(r)\frac{\zeta_N^j x}{(1 - \zeta_N^j x)^r} + R_j(x).
    \]
    Hence, multiplying both sides by $(1 - \zeta_N^j x)^k$, we obtain 
    \begin{equation}\label{A_j}
      A_j(x) = \sum_{r = 1}^{k}a_{N, j; Q}(r)\zeta_N^j x (1 - \zeta_N^j x)^{k-r} + (1 - \zeta_N^j x)^k R_j(x).
    \end{equation}
    Now, let $u = 1 - \zeta_N^j x$ and $f(u) = A_j(x)$. Expanding $f(u)$ in a Taylor series around $u = 0$, we write 
    \[
    f(u) = \sum_{m \geq 0}b_m u^m.
    \]
    With this substitution, \eqref{A_j} can be rewritten as
    \begin{align*}
      f(u) 
      &= \sum_{r = 1}^{k}a_{N, j; Q}(r) (1-u)u^{k-r} + O(u^k) \\
      &= (1-u)\left(
      a_{N, j; Q}(k) + a_{N, j; Q}(k-1)u + \cdots + a_{N, j; Q}(1)u^{k-1}
      \right) + O(u^k) \\
      &= a_{N, j; Q}(k) + \left(a_{N, j; Q}(k-1) - a_{N, j; Q}(k)\right)u + 
      \cdots + \left(a_{N, j; Q}(1) - a_{N, j; Q}(2)\right)u^{k-1} + O(u^k), 
    \end{align*}
    where $O(u^k)$ denotes terms of degree at most $k$.  That is, we have 
    \[
      b_m = 
      \begin{cases}
        a_{N, j; Q}(k) & \text{if } m = 0, \\
        a_{N, j; Q}(k-m) - a_{N, j; Q}(k-m+1) & \text{if } 1 \leq m \leq k-1,
      \end{cases}
    \]
    which yields 
    \[
    a_{N, j; Q}(r) = \sum_{m = 0}^{k-r}b_m.
    \]
    By direct calculation, we obtain $b_m = (-\zeta_N^{-j})^m \frac{A_j^{(m)}(\zeta_N^{-j})}{m!}$, which implies the desired result.
  \end{proof}

  \begin{proposition}
    Under the same assumptions and notations as in Proposition \ref{prop:a_N_j_r}, we have 
    \[
      c_{N, j; Q}(\ell) = \frac{1}{(k-1)!}\sum_{r = \ell}^{k}(-\zeta_N^{-j})^{k-r}A_j^{(k-r)}(\zeta_N^{-j})\binom{k-1}{r-1}\st{r}{\ell}
    \]
    for all $1 \leq \ell \leq k$ and $1 \leq j \leq \lfloor (N-1)/2 \rfloor$ such that $(j, N) = 1$.
  \end{proposition}
  \begin{proof}
    We consider the generating function of $c_{N, j; Q}(\ell)$ given by 
    \begin{equation}\label{GF_of_c_NjQ}
      \mathcal{C}_{N,j; Q}(X) \coloneq \sum_{\ell = 1}^{k}c_{N, j; Q}(\ell)X^{\ell-1}.
    \end{equation}
    Then, by the definition of $c_{N, j; Q}(\ell)$, \eqref{GF_of_c_NjQ} becomes 
    \begin{align*}
      \mathcal{C}_{N, j; Q}(X) 
      &= \sum_{\ell = 1}^{k}\left(\sum_{r = \ell}^{k}\frac{a_{N, j; Q}(r)}{(r-1)!}\st{r-1}{\ell-1}\right)X^{\ell-1} \\ 
      &= \sum_{r = 1}^{k}\frac{a_{N, j; Q}(r)}{(r-1)!}\sum_{\ell = 1}^{k}\st{r-1}{\ell-1}X^{\ell-1} \\
      &= \sum_{r = 1}^{k}\frac{a_{N, j; Q}(r)}{(r-1)!}(X)_{r-1} \\
      &= \sum_{r = 1}^{k}a_{N, j; Q}(r)\binom{X+r-2}{r-2}.
    \end{align*}
    Substituting \eqref{a_N_j_Q_r} into the above equation, we have 
    \begin{align*}
      \mathcal{C}_{N, j; Q}(X) 
      &= \sum_{r = 1}^{k}\left(
      \sum_{m = 0}^{k-r}(-\zeta_N^{-j})^m \frac{A_j^{(m)}(\zeta_N^{-j})}{m!}
      \right)\binom{X+r-2}{r-1} \\
      &= \sum_{m = 0}^{k-1}(-\zeta_N^{-j})^m \frac{A_j^{(m)}(\zeta_N^{-j})}{m!}\sum_{r = 1}^{k-m}\binom{X+r-2}{r-1} \\
      &= \sum_{m = 0}^{k-1}(-\zeta_N^{-j})^m \frac{A_j^{(m)}(\zeta_N^{-j})}{m!}\binom{X+k-m-1}{k-m-1}.
    \end{align*}
    Here, we have used the identity 
    \[
      \sum_{i = 0}^{n}\binom{x+i}{i} = \binom{x+n+1}{n}.
    \]
    Since the coefficients of $X^{\ell-1}$ in $\binom{X+M}{M}$ is given by $\frac{1}{M!}\st{M+1}{\ell}$, we have
    \begin{align*}
      c_{N, j; Q}(\ell) 
      &= \sum_{m = 0}^{k-\ell}\frac{(-\zeta_N^{-j})^m}{m!(k-m-1)!}A_j^{(m)}(\zeta_N^{-j})\st{k-m}{\ell} \\
      &= \frac{1}{(k-1)!}\sum_{m = 0}^{k-\ell}(-\zeta_N^{-j})^m A_j^{(m)}(\zeta_N^{-j})\binom{k-1}{m}\st{k-m}{\ell} \\
      &= \frac{1}{(k-1)!}\sum_{r = \ell}^{k}(-\zeta_N^{-j})^{k-r}A_j^{(k-r)}(\zeta_N^{-j})\binom{k-1}{r-1}\st{r}{\ell}.
      \qedhere
    \end{align*}
  \end{proof}

  By applying these identities, we reproduce the results in {\cite{nazaroglu2025}*{Example 1.2}}. 
  
  First, note that $\calU_{2, 2}(2; q)$ corresponds to the case $N = 2,\ k = 4$, and $Q(x) = x^2$. Thus, we obtain
  \begin{align*}
    c_{2; Q}(2) &= \frac{1}{3!}\left(
    Q^{(2)}(-1)\binom{3}{1}\st{2}{2} + Q^{(1)}(-1)\binom{3}{2}\st{3}{2} + Q(-1)\binom{3}{3}\st{4}{2}
    \right) = -\frac{1}{6} \\
    c_{2; Q}(4) &= \frac{1}{3!}Q(-1)\binom{3}{3}\st{4}{4} = \frac{1}{6}.
  \end{align*}
  Hence, we have 
  \begin{align*}
    \calU_{2, 2}(2; q) 
    &= \calU_{4; 2}(Q; q) \\
    &= c_{2; Q}(2)(4F_2(2\tau) - F_2(\tau)) + c_{2; Q}(4)(16F_4(2\tau) - F_4(\tau)) \\
    &= -\frac{2}{3}F_2(2\tau) + \frac{1}{6}F_2(\tau) + \frac{8}{3}F_4(2\tau) - \frac{1}{6}F_4(\tau) \\
    &= -\frac{G_4(\tau)}{6} + \frac{8}{3}G_4(2\tau) + \frac{G_2(\tau)}{6} - \frac{2}{3}G_2(2\tau) - \frac{1}{32}.
  \end{align*}

   Similarly, $\calU_{2, 2}(1; q)$ corresponds to the case $N = 3,\ k = 2$, and $Q(x) = x^2$. Thus, we obtain 
   \[
     c_{3, 1; Q}(1) = (-\zeta_3^{-1})^1 A_1^{(1)}(\zeta_3^{-1})\binom{1}{0}\st{1}{1} + A_1(\zeta_3^{-1})\binom{1}{1}\st{2}{1} = \frac{\sqrt{3}i}{9},\ 
     c_{3, 1; Q}(2) = A_1(\zeta_3^{-1})\binom{1}{1}\st{2}{2} = -\frac{1}{3}
   \]
   Here, for the principal character $\chi_N^{0}$ modulo $N$,
  \[
    F_k(\chi_N^{0}; \tau) = \sum_{d \mid N}\mu(d)d^{k-1}F_k(d\tau)
  \]
  holds, where $\mu(N)$ denotes the M\"{o}bius function. Thus, letting $\chi$ be the primitive character modulo $3$, we have
  \begin{align*}
    \calU_{2, 2}(1; Q) 
    &= \calU_{2; 3}(Q; q) \\
    &= c_{3, 1; Q}(1) G(\chi) F_1(\chi; \tau) + c_{3, 1; Q}(2) (G(\chi_3^0) F_2(\chi_3^0; \tau) + 6 G(\mathbf{1}) F_2(\mathbf{1}; 3\tau)) \\
    &= -\frac{1}{3}F_1(\chi; \tau) - \frac{1}{3}(3F_2(3\tau) - F_2(\tau) + 6F_2(3\tau)) \\
    &= -3F_2(3\tau) + \frac{F_2(\tau)}{3} - \frac{F_1(\chi; \tau)}{3} \\
    &= -3G_2(3\tau) + \frac{G_2(\tau)}{3} - \frac{G_1(\chi; \tau)}{3} - \frac{1}{18}.
  \end{align*}

    In addition, $\calU_{2, 2}(0; q)$ corresponds to the case $N = 4,\ k = 2$, and $Q(x) = x^2$. Thus, we obtain
    \[
      c_{4, 1; Q}(1) = (-\zeta_4^{-1})^1 A_1^{(1)}(\zeta_4^{-1})\binom{1}{0}\st{1}{1} + A_1(\zeta_4^{-1})\binom{1}{1}\st{2}{1} = 0,\ 
      c_{4, 1; Q}(2) = A_1(\zeta_4^{-1})\binom{1}{1}\st{2}{2} = - \frac{1}{4}
    \]
    Thus, letting $\psi$ be the primitive character modulo $4$, we have
    \begin{align*}
      \calU_{2, 2}(0; q) 
      &= \calU_{2; 4}(Q; q) \\
      &= c_{4, 1; Q}(1) G(\psi) F_1(\psi; \tau) + c_{4, 1; Q}(2)(G(\chi_4^0) F_2(\chi_4^0; \tau) + 4 G(\chi_2^0) F_2(\chi_2^0; 2\tau) + 8 G(\mathbf{1}) F_2(\mathbf{1}; 4\tau)) \\
      &= -\frac{1}{4}(4(2F_2(4\tau) - F_2(2\tau)) + 8 F_2(4\tau)) \\
      &= F_2(2\tau) - 4 F_2(4\tau) \\
      &= G_2(2\tau) - 4 G_2(4\tau) - \frac{1}{8}
    \end{align*}

  Finally, $\calU_{2, 2}(-1; q)$ corresponds to the case $N = 6,\ k = 2$, and $Q(x) = x^2$. Thus, we obtain
  \[
    c_{6, 1; Q}(1) = (-\zeta_6^{-1})^1 A_1^{(1)}(\zeta_6^{-1})\binom{1}{0}\st{1}{1} + A_1(\zeta_6^{-1})\binom{1}{1}\st{2}{1} = -\frac{\sqrt{3}i}{9},\ 
    c_{6, 2; Q}(2) = A_1(\zeta_6^{-1})\binom{1}{1}\st{2}{2} = -\frac{1}{3}.
  \]
  Thus, letting $\chi'$ be the character modulo $6$ induced by $\chi$, we have 
  \begin{align*}
    \calU_{2, 2}(-1; q) 
    &= \calU_{2; 6}(Q; q) \\
    &= c_{6, 1; Q}(1)(G(\chi')F_1(\chi'; \tau) + G(\chi)F_1(\chi; 2\tau)) \\
    &+ c_{6, 1; Q}(2)(G(\chi_6^0) F_2(\chi_6^0; \tau) + 2G(\chi_3^0) F_2(\chi_3^0; 2\tau) + 6G(\chi_2^0) F_2(\chi_2^0; 3\tau) + 12G(\mathbf{1}) F_2(\mathbf{1}; 6\tau)) \\
    &= -\frac{\sqrt{3}i}{9}(\sqrt{3}i(F_1(\chi; \tau) + F_1(\chi; 2\tau)) + \sqrt{3}i F_1(\chi; 2\tau)) \\
    &+ \frac{1}{3}\left(F_2(\tau) - 2F_2(2\tau) \right. \\
    &\left. - 3F_2(3\tau) + 6F_2(6\tau)) -2(F_2(2\tau) - 3F_2(6\tau)) -6(F_2(3\tau) - 2F_2(6\tau)) + 12F_2(6\tau)\right) \\
    &= -12F_2(6\tau) + 3F_2(3\tau) + \frac{4}{3}F_2(2\tau) - \frac{1}{3}F_2(\tau) + \frac{2}{3}F_1(\chi; 2\tau) + \frac{1}{3}F_{1}(\chi; \tau) \\
    &= -12G_2(6\tau) + 3G_2(3\tau) + \frac{4}{3}G_2(2\tau) - \frac{1}{3}G_2(\tau) + \frac{2}{3}G_1(\chi; 2\tau) + \frac{1}{3}G_{1}(\chi; \tau) - \frac{1}{2}.
  \end{align*}

  \section*{Acknowledgment}
  The author would like to express his sincere gratitude to Professor Yoshinori Yamasaki for valuable advice and helpful discussions.
  
  \bibliographystyle{abbrv}
  \bibliography{references}
\end{document}